\documentclass[12pt,amscd]{amsart}
\usepackage[all]{xy}
\usepackage{graphicx}
\usepackage{amsmath,amsxtra,amssymb,latexsym, amscd,amsthm}
\usepackage{indentfirst}
\usepackage[mathscr]{eucal}
\usepackage[pagebackref=true]{hyperref}

\newtheorem{thm}{Theorem}[section]
\newtheorem{cor}[thm]{Corollary}
\newtheorem{lem}[thm]{Lemma}

\theoremstyle{definition}
\newtheorem{defn}[thm]{Definition}
\newtheorem{exm}[thm]{Example}

\newtheorem{rem}[thm]{Remark}

\numberwithin{equation}{section}

\DeclareMathOperator{\NN}{\mathbb {N}}
\DeclareMathOperator{\ZZ}{\mathbb {Z}}
\DeclareMathOperator{\RR}{\mathbb {R}}

\DeclareMathOperator{\lk}{lk}
\DeclareMathOperator{\conv}{conv}
\DeclareMathOperator{\rp}{rp}

\DeclareMathOperator{\ord}{ord}
\DeclareMathOperator{\girth}{girth}

\DeclareMathOperator{\NP}{NP}
\DeclareMathOperator{\supp}{supp}

\DeclareMathOperator{\Inter}{Inter}
\DeclareMathOperator{\reg}{reg}
\DeclareMathOperator{\charr}{char}

\def\D {\Delta}

\def\a {\mathbf a}
\def\b {\mathbf b}

\def\m {\mathfrak m}
\def\F {\mathfrak F}

\def\k {\mathrm{k}}
\def\h {\widetilde{H}}
\def\c {\mathbf{c}}

\begin{document}

\title[Integral Closure] {Integral closure of small powers of edge ideals and their regularity}

\author{Nguyen Cong Minh}
\address{Department of Mathematics, Hanoi National University of Education, 136 Xuan Thuy, Hanoi,
	Vietnam}
\email{minhnc@hnue.edu.vn}
\author{Thanh Vu}
\email{vuqthanh@gmail.com}


\date{}

\dedicatory{Dedicated to Professor David Eisenbud on the occasion of his 75th birthday}
\commby{}
\maketitle
\begin{abstract}
    Let $I(G)$ be the edge ideal of a simple graph $G$ over a field $\k$. We prove that 
    $$\reg( \overline {I(G)^s}) = \reg(I(G)^s),$$
    for all $s \le 4$. Furthermore, we provide an example of a graph $G$ such that 
    $$\reg I(G)^s = \reg \overline{I(G)^s} = \reg I(G)^{(s)} = \begin{cases} 5 + 2s & \text{ if } \charr \k = 2 \\ 4 + 2s & \text{ if } \charr \k \neq 2, \end{cases}$$
    for all $s \ge 1.$
\end{abstract}

\maketitle

\section{Introduction}\label{sect_intro}
Let $S = \k[x_1, ..., x_n]$ be a standard graded polynomial ring over a field $\k$. Let $I = I(G)$ be the edge ideal of a simple graph $G$ on the vertex set $[n]$. For each $s \ge 1$, we have the following containments of ideals $I^s \subseteq \overline{I^s} \subseteq I^{(s)}$ where $\overline{I^s}$ is the integral closure of $I^s$ and $I^{(s)}$ is the $s$-th symbolic powers of $I$ (see Section \ref{sec_basic} for the definition of integral closure and symbolic powers). In \cite{MNPTV}, together with Nam, Phong, and Thuy, we laid out a general procedure for comparing regularity of monomial ideals and proved that $\reg I^s = \reg I^{(s)}$ for $s = 2, 3$. Subsequently, in \cite{MV1, MV2}, we proved a rigidity property of the regularity of intermediate ideals lying between $I^s$ and $I^{(s)}$. More precisely, for monomial ideals $J \subseteq K$, we define $\Inter(J,K)$ the set of monomial ideals $L$ such that $L = J + (f_1, \ldots, f_t)$ where $f_i$ are among minimal monomial generators of $K$. In this paper, we apply this procedure to prove that all ideals in $\Inter(I^s, \overline{I^s})$ have the same regularity for $s \le 4$.

\begin{thm}\label{main_thm}Let $I$ be the edge ideal of a simple graph $G$. For all $s \le 4$ and all intermediate ideals $J \in \Inter(I^s,\overline{I^s})$, we have
$$\reg J = \reg \overline{I^s} = \reg (I^s).$$
\end{thm} 

The asymptotic property of regularity of integral closure of powers $\reg \overline{I^s}$ was established in \cite{CHT} alongside the asymptotic property of regularity of powers $\reg I^s$ in \cite{CHT, Kod}. Nonetheless, to our knowledge, there was no explicit calculation of regularity of integral closure of powers of edge ideals in the case where $I$ is not normal. By \cite{SVV,OH}, $I$ is not normal, i.e. $\overline{I^s} \neq I^s$ for some $s > 1$ if and only if $G$ has an induced subgraph which is the disjoint union of two odd cycles. In the simplest case where $I$ is not normal, i.e. $G$ is an odd bicyclic graphs, Kumar and Kumar \cite{KK} prove that $\reg I^s = \reg \overline{I^s}$ for all $s$. The regularity of powers of edge ideals of bicyclic graphs is calculated in few cases \cite{CJNP, G}; though, these edge ideals are normal. The inequality $\reg \overline{I^s} \le \reg I^s$ is known in some cases \cite{BCDMS}. 

We realize that the Stanley-Reisner ideal $I_\Delta$ of a one-dimensional simplicial complex $\Delta$ provides plenty of examples of non-normal squarefree monomial ideals. Indeed, if $\Delta$ viewed as a graph, contains an induced subgraph that is isomorphic to the complete bipartite graph $K_{3,3}$, then $I_\Delta$ is non-normal. Using the arguments in \cite{MV2},  we deduce the following

\medskip

\noindent\textbf{Theorem \ref{dim1}.} Let $I = I_\Delta$ be the Stanley-Reisner ideal of a one-dimensional simplicial complex $\Delta$. Assume that $I$ is non-normal. Then $\girth \Delta \le 4$. Furthermore, for all $s \ge 1$, and all $J \in \Inter(I^s, \overline{I^s})$, we have 
$$\reg J = \reg I^s = \begin{cases} 3s & \text{ if } \girth \Delta = 3\\
2s + 1 & \text{ if } \girth \Delta = 4.\end{cases}$$

This paper is a continuation of \cite{MNPTV} where we analyse the degree complex $\Delta_\a(J)$ via the radical ideal $\sqrt{J:x^\a}$. When studying the degree complex $\Delta_\a(I^s)$ with $\a$ supports on two odd cycles, we discover the following example whose all powers have regularity depend on the characteristic of the base field.

\begin{exm} Let 
\begin{gather*}
I=(x_1y_1, x_2y_1, x_3y_1, x_7y_1, x_9y_1, x_1y_2, x_2y_2, x_4y_2, x_6y_2, x_{10}y_2,x_1y_3, x_3y_3, x_5y_3,x_6y_3,\\
x_8y_3, x_2y_4, x_4y_4, x_5y_4, x_7y_4, x_8y_4,x_3y_5, x_4y_5, x_5y_5, x_9y_5, x_{10}y_5,x_6y_6,x_7y_6, x_8y_6,x_9y_6,\\  x_{10}y_6,z_1z_2,z_2z_3,z_1z_3,z_4z_5,z_5z_6,z_4z_6) \subseteq k[x_1,\ldots,x_{10},y_1,\ldots,y_6,z_1,\ldots,z_6],
\end{gather*}
be the edge ideal of the disjoint union of two triangles and the graph in the Dalili and Kummini's example \cite{DK}. By Corollary \ref{cor_char_dependence}, for each $s \ge 1$, we have 
$$\reg I^s = \reg \overline{I^s} = \reg I^{(s)} = \begin{cases} 5 + 2s & \text{ when } \charr \k = 2 \\
4 + 2s &  \text{ when } \charr \k \neq 2. \end{cases}$$
In particular, $\reg I^s, \reg \overline{I^s}, \reg I^{(s)},$ and their asymptotic constants all depend on the characteristic of the base field.
\end{exm}
The example shows that the asymptotic linearity constant of $\reg I^s$ is not combinatorial. Note that, for edge ideals in the examples of Katzman \cite{Kat} and Dalili and Kummini \cite{DK}, the regularity of their second powers do not depend on the characteristic of the base field.

We now describe the idea of proof of Theorem \ref{main_thm}. For a monomial ideal $J$, an exponent $(\a,i) \in \NN^n \times \NN$ is called an extremal exponent of $J$ if $\reg S/J = |\a| + i$ and $\lk_{\Delta_\a(J)} F$ has a non-vanishing homology in degree $i-1$ for some face $F$ of $\Delta_\a(J)$ such that $F \cap \supp \a = \emptyset$. The steps to prove Theorem \ref{main_thm} ($s = 3$ and $s = 4$) are:
\begin{enumerate}
    \item Let $J_1 \subseteq J_2 \in \Inter(I^s,\overline{I^s})$ be intermediate ideals. Let $(\a,i)$ be an extremal exponent of $J_2$. We prove that if $\sqrt{J_2:x^\a} \neq \sqrt{J_1:x^\a}$, then there exists a variable $x_t$ such that $x_t \in \sqrt{J_2:x^\a}$ and $x_t \notin \supp x^\a$. By induction, Lemma \ref{red0}, and Lemma \ref{extremal_red} we deduce that $\reg J_2 \le \reg J_1.$
    \item For each form $x^\a \notin I^s$, for which $\Delta_\a(I^s) \neq \Delta_\a(\overline{I^s})$, using Lemma \ref{red0} and induction, we reduce to the case where $G$ is the disjoint union of a smaller subgraph and odd cycles. We then use \cite[Theorem 1.1]{NV1} to establish the reverse inequality $\reg I^s \le \reg \overline{I^s}$.
\end{enumerate}

Since we do not have a mixed sum type formula for the regularity of the integral closure of powers, we need to consider disconnected graphs in our arguments. The main obstruction to carry out this comparison for higher powers is that when $s \ge 5$, the difference between $\sqrt{\overline{I^s}:x^\a}$ and $\sqrt{I^s:x^\a}$ are hard to describe.

Now we explain the organization of the paper. In Section \ref{sec_basic}, we recall some notation and basic facts about Castelnuovo-Mumford regularity, the integral closure, and the degree complexes of monomial ideals. In Section \ref{sec_proof}, we prove Theorem \ref{main_thm}. In section \ref{sec_dim1}, we compute the regularity of intermediate ideals in $\Inter(I_\Delta^s,\overline{I_\Delta^s})$ for Stanley-Reisner ideals of one-dimensional simplicial complexes. In Section \ref{sec_ex}, we provide examples of graphs whose all powers have regularity depend on the characteristic of the base field.

\section{Castelnuovo-Mumford regularity, integral closure, and degree complexes}\label{sec_basic}
In this section, we recall some definitions and properties concerning Castelnuovo-Mumford regularity, the degree complexes of monomial ideals, symbolic powers of squarefree monomial ideals, and integral closure of powers of edge ideals. See \cite{BH, D, E, S, V} for more details.

Throughout the paper, let $S = \k[x_1,...,x_n]$ be a standard graded polynomial ring over a field $\k$. For statements without proper citations, we refer to \cite{MNPTV} for proofs.

\subsection{Graphs and their edge ideals}

Let $G$ denote a finite simple graph over the vertex set $V(G)=[n] = \{1,2,\ldots,n\}$ and the edge set $E(G)$. For a vertex $x\in V(G)$, let the neighbours of $x$ be the subset $N_G(x)=\{y\in V(G)~|~ \{x,y\}\in E(G)\}$, and set $N_G[x]=N_G(x)\cup\{x\}$. For a subset $U$ of the vertices set $V(G)$, $N_G(U)$ and $N_G[U]$ are defined by $N_G(U)=\cup_{u\in U}N_G(u)$ and $N_G[U]=\cup_{u\in U}N_G[u]$. If $G$ is fixed, we shall use $N(U)$ or $N[U]$ for short.

An independent set in $G$ is a set of pairwise non-adjacent vertices.

A subgraph $H$ is called an induced subgraph of $G$ if for any vertices $u,v\in V(H)\subseteq V(G)$ then $\{u,v\}\in E(H)$ if and only if $\{u,v\}\in E(G)$.

An induced matching is a subset of the edges that do not share any vertices and it is an induced subgraph. The induced matching number of $G$, denoted by $\mu(G)$, is the largest size of an induced matching in $G$.

A $m$-cycle in $G$ is a sequence of $m$ distinct vertices $1,\ldots, m\in V(G)$ such that $\{1,2\},\ldots, \{m-1,m\}, \{m,1\}$ are edges of $G$. The girth of $G$, denoted $\girth(G)$ is the size of a smallest induced cycle in $G$.

A clique in $G$ is a complete subgraph of $G$. We also call a clique of size $3$ a triangle.

A graph $G$ is bipartite if $[n]$ can be partition into two disjoint subsets $[n] = U\cup V$ such that every edge connects a vertex in $U$ to one in $V$.

The edge ideal of $G$ is defined to be
$$I(G)=(x_ix_j~|~\{i,j\}\in E(G))\subseteq S.$$
For simplicity, we often write $i \in G$ (resp. $ij \in G$) instead of $i \in V(G)$ (resp. $\{i,j\} \in E(G)$). By abuse of notation, we also call $x_ix_j \in I(G)$ an edge of $G$.

\subsection{Simplicial complexes and Stanley-Reisner correspondence} 
Let $\Delta$ be a simplicial complex on $[n]=\{1,\ldots, n\}$ that is a collection of subsets of $[n]$ closed under taking subsets. We put $\dim F = |F|-1$, where $|F|$ is the cardinality of $F$. The dimension of $\Delta$ is $\dim \Delta = \max \{ \dim F \mid F \in \Delta \}$.  The set of its maximal elements under inclusion, called by facets, is denoted by $\F(\Delta)$.

A simplicial complex $\D$ is called a cone over $x\in [n]$ if $x\in B$ for any $B\in \F(\Delta)$. If $\D$ is a cone, it is acyclic (i.e., has vanishing reduced homology).

For a face $F\in\Delta$, the link of $F$ in $\Delta$ is the subsimplicial complex of $\Delta$ defined by
$$\lk_{\Delta}F=\{G\in\Delta \mid  F\cup G\in\Delta, F\cap G=\emptyset\}.$$

For each subset $F$ of $[n]$, let $x_F=\prod_{i\in F}x_i$ be a squarefree monomial in $S$. We now recall the Stanley-Reisner correspondence

\begin{defn}For a squarefree monomial ideal $I$, the Stanley-Reisner complex of $I$ is defined by
$$ \Delta(I) = \{ F \subset [n] \mid x_F \notin I\}.$$

For a simplicial complex $\Delta$, the Stanley-Reisner ideal of $\Delta$ is defined by
$$I_\Delta = (x_F \mid  F \notin \Delta).$$
The Stanley-Reisner ring of $\Delta$ is $\k[\Delta] =  S/I_\Delta.$
\end{defn}
From the definition, it is easy to see the following:
\begin{lem}\label{cone} Let $I, J$ be squarefree monomial ideals of  $S = \k [x_1,\ldots, x_n]$. Then 
\begin{enumerate}
    \item $\Delta(I)$ is a cone over $t \in [n]$ if and only if $x_t$ does not divide any minimal generator of $I$.
    \item $I \subseteq J$ if and only if $\Delta(I) \supseteq \Delta(J)$.
    \item $\Delta(I + J) = \Delta(I) \cap \Delta(J).$
    \item $\Delta(I \cap J) = \Delta(I) \cup \Delta(J).$
\end{enumerate}
\end{lem}

\subsection{Castelnuovo-Mumford regularity}\label{subsection_reg} 
Let $\m = (x_1,\ldots, x_n)$ be the maximal homogeneous ideal of $S$. For a finitely generated graded $S$-module $L$, let
$$a_i(L)=
\begin{cases}
\max\{j\in\ZZ \mid H_{\m}^i(L)_j \ne 0\} &\text{ if  $H_{\m}^i(L)\ne 0$}\\ 
-\infty &\text{ otherwise,}
\end{cases}
$$
where $H^{i}_{\m}(L)$ denotes the $i$-th local cohomology module of $L$ with respect to $\m$. Then, the Castelnuovo-Mumford regularity (or regularity for short) of $L$ is defined to be
$$\reg(L) = \max\{a_i(L) +i\mid i = 0,\ldots, \dim L\}.$$

For a non-zero and proper homogeneous ideal $J$ of $S$ we have $\reg(J)=\reg(S/J)+1$. 

We frequently use the result on the regularity of powers of mixed sum \cite[Theorem 1.1]{NV1} in this paper, so we state a part of it here for ease of reference.
\begin{thm}\label{thm_mixed_sum} Let $I \subseteq R$ and $J \subseteq S$ be proper monomial ideals in standard graded polynomial rings over a field $\k$. Let $P$ be the sum of the extensions of $I$ and $J$ to $R \otimes_\k S$. Then
$$\reg P^s = \max_{i\in [1,s-1], j\in [1,s]} \left ( \reg I^i + \reg J^{s-i}, \reg I^j + \reg J^{s-j+1} - 1\right ).$$
In particular, 
$$\reg P^s \ge \reg I^j + \reg J^{s-j+1} - 1, \text{ for } j \in [1,s].$$
\end{thm}

\subsection{Integral closure of monomial ideals}
\begin{defn} Let $I$ be a monomial ideal of $S$. The exponent set of $I$ is $E(I) = \{\a \in \NN^n | x^\a \in I\}$. The Newton polyhedron of $I$ is 
$$\NP(I) = \conv (\a \in \NN^n | x^\a \in  I ),$$
the convex hull of the exponent set of $I$ in $\RR^n$.
\end{defn}
The following characterization of the integral closure of $I$ in term of the Newton polyhedron of $I$ is well-known (see e.g. \cite{E}).
\begin{lem}\label{lem_NP} Let $I$ be a monomial ideal of $S$. Then the integral closure of $I$ is a monomial ideal with the exponent
set $E(\overline {I}) = \NP(I) \cap \ZZ^r.$
\end{lem}
\begin{defn}A monomial ideal $I$ of $S$ is said to be normal if $\overline{I^s} = I^s$ for all $s \ge 1$.
\end{defn}

For a monomial $f$ in $S$, the support of $f$, denoted $\supp(f)$, is the set of all indices $i \in [n]$ such that $x_i|f$. For a monomial ideal $J$ of $S$ and a subset $V$ of $[n]$, the restriction of $J$ to $V$, denoted $J_V$ is:
$$J_V = (f \mid f \text{ is a minimal generator of J} \text{ such that } \supp f \subseteq V).$$
We have

\begin{cor}\label{restriction} Let $I$ be a monomial ideal. Then 
$$\overline{(I_V)} = (\overline{I})_V.$$
\end{cor}
\begin{proof} Follows from Lemma \ref{lem_NP}.
\end{proof}

\begin{cor}\label{restriction_intermediate} Let $I$ be a monomial ideal. Let $J \in \Inter(I^s,\overline{I^s})$ be an intermediate ideal. Let $V \subseteq [n]$ be a subset. Then $J_V \in \Inter(I_V^s, \overline{I_V^s}).$
\end{cor}
\begin{proof}Follows from the definition and Corollary \ref{restriction}.
\end{proof}

For a monomial $f$ of $S$ and $i \in [n]$, $\deg_i(f) = \max (t \mid x_i^t \text{ divides } f)$ denotes the degree of $x_i$ in $f$. For a monomial ideal $I$, the degree of $x_i$ in $I$ is defined by 
$$\rho_i(I)=\max\{\deg_{i}(u) \mid u\text{ is a minimal monomial generator of } I\}.$$
We have the following property of minimal generators of integral closure. 
\begin{lem}\label{partial_degree} Let $I$ be a monomial ideal. Then $\rho_j(\overline{I}) \le \rho_j(I),$ for all $j = 1, \ldots, n$. In particular, if $I$ is squarefree, then $\rho_j(\overline{I^s}) \le s$ for all $j = 1, \ldots, n$.
\end{lem}
\begin{proof}Let $x^\a$ be a minimal generator of $\overline{I}$. By Lemma \ref{lem_NP}, there exists non-negative real numbers $c_1, ..., c_m$ with $\sum_{i=1}^m c_i = 1$ and minimal generators $x^{\b_1}, ..., x^{\b_m}$ of $I$ such that 
$$a_j \ge \sum_{i=1}^m c_i b_{ij}$$
for all $j$. Since $x^\a$ is minimal, this implies that $a_j = \lceil \sum_{i=1}^m c_i b_{ij} \rceil$ for all $j$. Since $b_{ij} \le \rho_j(I)$ and $\sum_{i=1}^m c_i = 1$, this implies that $\sum_{i=1}^m c_i b_{ij} \le \rho_j(I)$. Thus $a_j \le \rho_j(I)$. The conclusion follows.
\end{proof}

We now describe the minimal generators of the integral closure of powers of edge ideals of simple graphs. 
     
\begin{thm}\label{expansion_edge} Let $I = I(G)$ be the edge ideal of a simple graph $G$. Let $f$ be a minimal generator of $\overline{I(G)^s}$. Then there exist odd cycles $C_1, ..., C_{2a}$, and edges $e_1, ...,e_b$ of $G$ such that
$$f = C_1 \cdots C_{2a} \cdot e_1 \cdots e_b,$$
where $(|C_1| + \cdots + |C_{2a}|)/2 + b = s.$ Furthermore, the odd cycles $C_i$ can be chosen so that $N(C_i) \cap C_j = \emptyset$ for all $i \neq j \in \{1, ..., 2a\}$.
\end{thm}
\begin{proof} For each pair $C_1, C_2$ of induced disjoint odd cycles, denote $m_{C_1,C_2} = x_{C_1} \cdot x_{C_2}$. Denote $P$ the set of all pairs of disjoint odd cycles in $G$. By \cite[Proposition 8.7.19]{V}, we have 
$$\overline{\mathcal{R}(It)} = S[It,m_{C_i,C_j}t^{m_{i,j} /2} \mid (C_i,C_j) \in P],$$
where $m_{i,j} = \deg m_{C_i,C_j}$. The conclusion follows by taking the graded component of degree $s$ and a remark that if $N(C_i) \cap C_j \neq \emptyset$, then $x_{C_i}\cdot x_{C_j} \in I^{u}$ with $u= m_{i,j}/2.$
\end{proof}

\subsection{Symbolic powers} 
Let $I$ be a non-zero and proper homogeneous ideal of $S$. Let $\{P_1,\ldots,P_r\}$ be the set of the minimal prime ideals of $I$. Given a positive integer $s$, the $s$-th symbolic power of $I$ is defined by
$$I^{(s)}=\bigcap_{i=1}^r I^sS_{P_i}\cap S.$$

\subsection{Degree complexes}
For a monomial ideal $I$ in $S$, Takayama in \cite{T} found a combinatorial formula for $\dim_\k H_\m^i(S/I)_\a$ for all $\a\in\ZZ^n$ in terms of certain simplicial complexes which are called degree complexes. For every $\a = (a_1,\ldots, a_n) \in \ZZ^n$ we set $G_\a = \{i\mid \ a_i < 0\}$ and write $x^{\a} = \Pi_{j=1}^n x_j^{a_j}$. Thus, $G_\a =\emptyset$ whenever $\a \in \NN^n$. The degree complex $\D_\a(I)$ is the simplicial complex whose faces are $F \setminus G_\a$, where $G_\a\subseteq F\subseteq [n]$, so that for every minimal generator $x^\b$ of $I$ there exists an index $i \not\in F$ with $a_i < b_i$. It is noted that $\D_\a(I)$ may be either the empty complex or $\{\emptyset\}$ and its vertex set may be a proper subset of $[n]$. The next lemma is useful to compute the regularity of a monomial ideal in terms of its degree complexes.

\begin{lem}\label{Key0}
Let $I$ be a monomial ideal in $S$. Then
\begin{multline*}
\reg(S/I)=\max\{|\a|+i~|~\a\in\NN^n,i\ge 0,\h_{i-1}(\lk_{\D_\a(I)}F;\k)\ne 0\\ \text{ for some $F\in \D_\a(I)$ with $F\cap \supp \a=\emptyset$}\}.
\end{multline*}
In particular, if $I=I_\D$ is the Stanley-Reisner ideal of a simplicial complex $\D$ then
$\reg(\k[\D])=\max\{i \mid \h_{i-1}(\lk_{\D}F;\k)\ne 0\text{ for some }F\in \D\}$.
\end{lem}

\begin{rem}\label{rem_T} Let $I$ be a monomial ideal in $S$ and a vector $\a\in \NN^n$. In the proof of \cite[Theorem 1]{T}, the author showed that if there exists $j\in [n]$ such that $a_j\ge \rho_j(I)$ then $\D_\a(I)$ is either a cone over $\{j\}$ or the void complex. 
\end{rem}

\begin{defn}\label{exdef} Let $I$ be a monomial ideal in $S$. A pair $(\a,i) \in \NN^n\times\NN$ is called {\it an extremal exponent of $I$} if $\reg(S/I) = |\a| + i$ as in Lemma \ref{Key0}. 
\end{defn}

\begin{rem}\label{rem_extremal_set} We sometime call $\a$ instead of $(\a,i)$ an extremal exponent of $I$. Let $\a$ be an extremal exponent of $I$. Then $x^\a \notin I$ and $\Delta_{\a}(I)$ is not a cone over $t$ with $t\in\supp\a$. In particular, by Remark \ref{rem_T}, $\a$ belongs to the finite set
	$$\Gamma(I)=\{\a\in\NN^n~|~ a_j<\rho_j\text{ for all } j=1,\ldots,n\}.$$ 
\end{rem}

Furthermore, we have the following interpretation of the degree complex $\Delta_\a(I)$.

\begin{lem}\label{Key1}
Let $I$ be a monomial ideal in $S$ and $\a\in\NN^n$. Then
$$I_{\Delta_{\a}(I)}=\sqrt{I : x^\a}.$$
\end{lem}
\begin{rem}\label{rem_mingens_degree_complex} Let $\a$ be an extremal exponent of $I$. By Lemma \ref{cone}, Lemma \ref{Key1}, and Remark \ref{rem_extremal_set}, for each $t \in \supp \a$, there exists a minimal generator $f$ of $\sqrt{I:x^\a}$ such that $x_t | f$.
\end{rem}

 We have the following inequality on the regularity of restriction of a monomial ideal.

\begin{lem}\label{restriction_in} Let $I$ be a monomial ideal and $x$ be a variable. Then 
$$\reg (I,x) \le \reg I.$$
\end{lem}
\begin{proof}
See \cite[Corollary 4.8]{CHHKTT} or \cite[Lemma 2.14]{MNPTV}.
\end{proof}

\begin{lem}\label{restriction_inq} Let $J$ be a monomial ideal in $S$ and $V \subseteq [n]$. We have
$$\reg (J_V) \le \reg (J).$$
In particular, for a monomial ideal $I \subseteq S$ then
$$\reg \overline{I_V^s} \le \reg I^s.$$
\end{lem}
\begin{proof}
 Let $\{t,\ldots,n\} = [n] \setminus V$. Then, $J_V + (x_t,...,x_n) = J + (x_t,...,x_n).$ The conclusion follows from Lemma \ref{restriction_in} and the fact that $x_t,\ldots, x_n$ is a regular sequence with respect to $S/J_V$.
 
 The last statement follows from the first statement and Lemma \ref{restriction}.
\end{proof}

The following lemma is essential to using the induction method in studying the regularity of a monomial ideal.

\begin{lem}\label{red0} Let $I$ be a monomial ideal and the pair $(\a,i) \in \NN^n\times\NN$ be its extremal exponent. If $x$ is a variable that appears in $\sqrt{I:x^\a}$ and $x \notin \supp \a$, then $$\reg (I) = \reg (I,x).$$
\end{lem}

For comparing regularity, we have the following useful lemma.
\begin{lem}\label{extremal_red} Let $I, J$ be proper monomial ideals of $S$. Let $(\a,i)$ be an extremal exponent of $I$. If $\Delta_\a(I) = \Delta_\a(J)$, then $\reg I \le \reg J$. In particular, if $J \subseteq I$ and $\Delta_\a(I) = \Delta_\a(J)$ for all exponent $\a\in\NN^n$ such that $x^\a \notin I$ then $\reg I \le \reg J.$
\end{lem}

For a monomial $f$ in $S$, the radical of $f$ is defined by $\sqrt{f} = \prod_{i \in \supp f} x_i$. We have the following simple observation on the generators of the radical of colon ideal. 

\begin{lem}\label{radical_colon} Let $I$ be a monomial ideal in $S$ generated by the monomials $f_1, ..., f_s$ and $\a \in \NN^n$. Then $\sqrt{I:x^\a}$ is generated by $\sqrt{f_1/\gcd(f_1, x^\a)}, ..., \sqrt{f_s/\gcd(f_s,x^\a)}$.
\end{lem}

Let $\a \in \NN^n$ be an exponent. To study the radical ideal of the form $\sqrt{I^s:x^\a}$, we introduce the notion of $I-\a$-radical power. For a monomial $f \in S$, the $I-\a$-radical power of $f$ is
$$\rp_\a^I(f) = \max (t \mid f^m x^\a \in I^t \text{ for some } m > 0).$$
If $f \in I$ then $\rp_\a^I(f) = \infty$. When fixing an ideal $I$, we omit $I$ in the notation, and simply write $\rp_\a(f)$. 

Assume now that $I = I(G)$ is the edge ideal of a graph. The following description of generators of $\sqrt{I^s:x^\a}$ helps to simplify our arguments later on. The $I$-order of $f$ is defined by 
$$\ord_I(f) = \max (t \mid f \in I^t).$$
From the definition, it is clear that if $g|f$, then $\ord_I(g) \le \ord_I(f)$.

\begin{lem}\cite[Lemma 2.16]{MV2}\label{criterion_in_power}  Let $F$ be an independent set of $G$, and $\a \in \NN^n$ an exponent. Assume that
\begin{equation}
  \sum_{j\in N(F)} a_j + \ord_I \left ( \prod_{u \notin N[F]} x_u^{a_u} \right ) \ge s, \label{eq_in_power} 
\end{equation}
then $x_F \in \sqrt{I^s:x^\a}$. Conversely, if $x_F$ is a minimal generator of $\sqrt{I^s:x^\a}$ then \eqref{eq_in_power} holds.
\end{lem}

We have the following property of the radical powers.
\begin{lem}\label{lem_radical_power} Let $I = I(G)$ be an edge ideal of a simple graph $G$. Let $\a, \b \in \NN^n$ be exponents, and $f$ a monomial in $S$. Then
\begin{enumerate}
    \item $\rp_\a(f) \ge \ord_I(x^\a)$.
    \item $\rp_{\a + \b}(f)  \ge \rp_\a(f) + \rp_\b(f)$.
    \item If $\supp f \cup \supp \a \supseteq \supp C$, where $C$ is an odd cycle of length $2\ell+1$, then $\rp_\a(f) \ge \ell$. Furthermore, if $\supp f \cap \supp C \neq \emptyset$ then $\rp_\a(f) \ge \ell+1.$
    \item Let $C_1, C_2$ be disjoint odd cycles of length $2 \ell_1 +1, 2\ell_2+1$ respectively. If $\supp f \cup \supp \a \supseteq (\supp C_1 \cup \supp C_2)$, and $\supp f \cap (\supp C_1 \cup \supp C_2) \neq \emptyset$ then $\rp_\a(f) \ge \ell_1 + \ell_2 + 1.$
    \item Let $F = \supp f$. Then 
    $$\rp_\a(f) \ge \sum_{j\in N(F)} a_j + \ord_I \left ( \prod_{u \notin N[F]} x_u^{a_u} \right ).$$
\end{enumerate}
\end{lem}
\begin{proof}
Parts (1) and (2) follow from the definition. For part (3), it suffices to note that $\ord_I(x_C) = \ell$ and $\ord_I(x_ix_C) = \ell+1$ for any $i \in \supp C$. Part (4) follows from Part (2) and Part (3). Part (5) follows from Lemma \ref{criterion_in_power} and a remark that if $F$ is not an independent set then $\rp_\a(f) = \infty$.
\end{proof}

\section{Proof of Theorem \ref{main_thm}}\label{sec_proof}

Let $G$ be a simple graph with vertex set $[n]$ and edge set $E(G)$. We assume that $G$ has no isolated vertices; $G$ might be disconnected. Let $I=I(G)$ be the edge ideal of $G$. In this section, we will prove Theorem \ref{main_thm}. By Theorem \ref{expansion_edge}, $\overline{I} = I$ and $\overline{I^2} = I^2$, we have two cases $s = 3$ and $s=4$. We start with a general formula for the radical of colon ideal in the case it contains a set of 'good variables'.

\begin{lem}\label{rad_good_colon} Let $I$ be the edge ideal of a simple graph $G$. Let $\a \in \NN^n$ be an exponent such that $x^\a \notin I^s$. Assume that $\a$ has a decomposition $\a = \b + \c$ satisfying the following conditions:
\begin{enumerate}
    \item $x_i \in \sqrt{I^s:x^\a}$ for all $i \in N[\supp \b]$,
    \item $N[\supp \b] \cap \supp \c = \emptyset$.
\end{enumerate}
Let $H$ be the restriction of $G$ to $[n] \setminus N[\supp \b]$. Denote $t = s - \ord_I(x^\b)$. Then, 
\begin{equation}\label{eq_rad_colon}
    \sqrt{I^s:x^\a} = \left ( x_i \mid  i \in N[\supp \b] \right ) + \sqrt{I(H)^{t}:x^\c}.
\end{equation}
In particular, if $(\a,i)$ is an extremal exponent of $I^s$ then
\begin{equation}\label{eq_reg_inequality}
    \reg (I^s) \le \reg I(H)^t + |\b|.
\end{equation}
\end{lem}
\begin{proof} First, we prove the equality \eqref{eq_rad_colon}. From the assumption, it is clear that the left hand side contains the right hand side. Let $V = [n] \setminus N[\supp \b]$. It suffices to prove that if $f$ is a minimal generator of  $\sqrt{I^s:x^\a}$ with $\supp f \subseteq V$, then $f \in \sqrt{I(H)^t:x^\c}$. If $f \in I$, then by definition, $f \in I(H) \subseteq \sqrt{I(H)^t:x^\c}$. Thus, we may assume that $F = \supp f$ is an independent set. By Lemma \ref{criterion_in_power}, we have 
$$\sum_{j \in N(F)} a_j + \ord_I( h )\ge s, \text{ where } h:= \prod_{u \notin N[F]} x_u^{a_u}.$$

Since $N[\supp \b] \cap \supp \c = \emptyset$, $N[\supp \c] \cap \supp \b = \emptyset$. In particular, $N[F] \cap \supp \b = \emptyset$. Hence,
\begin{equation}\label{eq_3_1_2}
   \sum_{j\in N(F) \cap V} c_j = \sum_{j \in N(F)} a_j. 
\end{equation}
Let
$$h_1 := \prod_{u\in V\setminus N[F]} x_u^{c_u} = \prod_{u \in V \setminus N[F]} x_u^{a_u}.$$
Hence, $h = x^\b h_1$. Since $N[\supp \b] \cap \supp h_1 = \emptyset$, $\ord_I(h) = \ord_I(x^\b) + \ord_I(h_1)$. Thus
$$\sum_{j \in N(F) \cap V} c_j + \ord_I(h_1) = \sum_{j \in N(F)} a_j + \ord_I(h) - \ord_I(b) \ge t.$$
By Lemma \ref{criterion_in_power}, $f \in \sqrt{I(H)^t : x^\c}$ as required.

We now prove \eqref{eq_reg_inequality}. By \eqref{eq_rad_colon}, $\Delta_\a(I^s) = \Delta_\c(I(H)^t)$. Since $(\a,i)$ is an extremal exponent of $I^s$, $x^\a \notin I^s$, hence $t \ge 1$ and $x^\c \notin I^t$. By Lemma \ref{Key0}, 
$$\reg I^s = |\a| +i+1 = |\b| + |\c| + i+1 \le |\b| + \reg I(H)^t.$$
The conclusions follow.
\end{proof}

\begin{defn} A decomposition $\a = \b + \c$, where $\b \neq 0$ satisfying the conditions of Lemma \ref{rad_good_colon} is called a {\em good decomposition} of $\a$ with respect to $I^s$. We also call a decomposition $x^\a = x^\b \cdot x^\c$ good if the corresponding decomposition on the exponents $\a = \b + \c$ is.
\end{defn}
In proving the inequality $\reg I^s \le \reg \overline{I^s}$, we see that the extremal exponents $\a$ of $I^s$ with $\Delta_\a(I^s) \neq \Delta_\a(\overline{I^s})$ usually have good decompositions. We then use induction and Theorem \ref{thm_mixed_sum} to establish the required inequality.
\begin{lem}\label{lem_inequality_special}Let $I$ be the edge ideal of a simple graph $G$. Fix a power $s \ge 2$. Let $\a \in \NN^n$ be an extremal exponent of $I^s$. Assume that $\a$ has a good decomposition $\a = \b + \c$. Let $H$ be the restriction of $G$ to $[n] \setminus N[\supp \b]$. Let $K$ be the restriction of $G$ to a proper subset of $\supp \b$. Assume that $\reg I(K)^{s-t+1} \ge |\b| + 1$, where $t = s - \ord_I(x^\b)$ and that $L = I(H) + I(K)$ satisfies $\reg L^s \le \reg \overline{L^s}$. Then $\reg I^s \le \reg \overline{I^s}$.
\end{lem}
\begin{proof}By Lemma \ref{rad_good_colon}, $\reg I^s \le \reg I(H)^s + |\b|$. By Theorem \ref{thm_mixed_sum}, 
$$\reg L^s \ge \reg I(H)^t + \reg I(K)^{s-t+1} - 1 \ge \reg I(H)^t + |\b| \ge \reg I^s.$$
By assumption $\reg L^s \le \reg \overline{L^s}$, hence 
$$\reg I^s \le \reg L^s \le \reg \overline{L^s} \le \overline{I^s},$$
where the last inequality follows from Lemma \ref{restriction_inq}.
\end{proof}

We now proceed to the proof of Theorem \ref{main_thm} for $s = 3$. First, we have
 
\begin{lem}\label{lem_colon_3} Let $J_1 \subseteq J_2$ be any two intermediate ideals in $\Inter(I^3, \overline{I^3})$. Let $\a\in \NN^n$ be an exponent such that $x^\a \notin J_2$. Then 
$$\sqrt{J_2:x^\a} = \sqrt{J_1:x^\a}.$$
In particular, $\Delta_\a(J_2) = \Delta_\a(J_1)$.
\end{lem}
\begin{proof} Since $J_1 \subseteq J_2$, $\sqrt{J_1:x^\a} \subseteq \sqrt{J_2:x^\a}$. By Lemma \ref{radical_colon}, it suffices to prove that if $f = \sqrt{P/\gcd(P,x^\a)}$ for a minimal generator $P$ of $J_2$ then $f \in \sqrt{J_1:x^\a}$. We may assume that $P \notin J_1$. In particular, $P \notin I^3$. By Theorem \ref{expansion_edge}, there exist two disjoint triangles $C_1$ and $C_2$ such that $P = x_{C_1} \cdot x_{C_2}$. Since $f = \sqrt{P/\gcd(P,x^\a)}$, we have 
\begin{equation}\label{eq_lem_1_1}
    \supp f \subseteq \supp P \subseteq \supp f \cup \supp \a.
\end{equation}
Since $f\neq 1$, $\supp f \cap \supp P \neq \emptyset$. By Lemma \ref{lem_radical_power}, 
$$\rp_\a(f) \ge 3 \implies f \in \sqrt{I^3:x^\a} \subseteq \sqrt{J_1:x^\a}.$$
The last statement follows from Lemma \ref{Key1}.
\end{proof}

\begin{thm}\label{third_pow} Let $I$ be the edge ideal of a simple graph $G$. Let $J \in \Inter(I^3, \overline{I^3})$ be an intermediate ideal. Then 
$$\reg J = \reg \overline {I^3}  = \reg I^3.$$
\end{thm}
\begin{proof} Let $J_1 \subseteq J_2$ be any two intermediate ideals in $\Inter(I^3, \overline{I^3})$. By Lemma \ref{lem_colon_3} and Lemma \ref{extremal_red}, $\reg J_2 \le \reg J_1.$ 

It remains to prove that $\reg I^3 \le \reg \overline{I^3}$. We prove by induction on $n = |V(G)|$. By Theorem \ref{expansion_edge}, we may assume that $n \ge 6$, as if $n \le 5$, then $\overline{I^3} = I^3$. Let $(\a,i)\in\NN^n\times\NN$ be an extremal exponent of $I^3$. If $x^\a \notin \overline{I^3}$, by Lemma \ref{lem_colon_3} and Lemma \ref{extremal_red}, $\reg I^3 \le \reg \overline {I^3}$. Now assume that $x^\a \in \overline{I^3}$. By Theorem \ref{expansion_edge}, there exist two triangles $C_1$, $C_2$ in $G$ such that $N[C_1] \cap C_2 = \emptyset$ and $x^\a = x_{C_1} \cdot x_{C_2} \cdot f$, for some monomial $f \in S$. Since $x^\a \notin I^3$, 
\begin{equation}\label{eq_3_1}
    N[C_1 \cup C_2] \cap \supp f = \emptyset.
\end{equation} 
Furthermore, $x_i \in \sqrt{I^3:x^\a}$ for all $i \in N[C_1 \cup C_2]$. Thus, the decomposition $x^\a = (x_{C_1} x_{C_2}) \cdot f$ is a good decomposition of $I^3$. Let $K$ be a maximum induced matching of $C_1 \cup C_2$. Then $\reg I(K)^3 = 7 =  |\b| + 1$. Then $L := I(H) + I(K)$ is the restriction of $I$ to $K \cup H$ which is a proper subset of $[n]$. By induction $\reg L^3 \le \reg \overline{L^3}$. By Lemma \ref{lem_inequality_special}, $\reg I^3 \le \reg \overline{I^3}.$ The conclusion follows.
\end{proof}

For the fourth power, we first prove a relation on radical of colon ideal. We also use the following notation. Let $\a \in \NN^n$ be an exponent and $V \subset [n]$. The restriction of $\a$ to $V$ is an exponent $\b$ such that $b_i = a_i$ if $i \in V$ and $b_i = 0$ if $i \notin V$. 

\begin{lem}\label{lem_colon_4} Let $J_1 \subseteq J_2$ be intermediate ideals in $\Inter(I^4, \overline{I^4})$. Let $\a$ be an exponent such that $x^\a \notin J_2$. Assume that $\sqrt{J_2 : x^\a} \neq \sqrt{J_1:x^\a}$. Let $f$ be a minimal generator of $\sqrt{J_2:x^\a}$ such that $f \notin \sqrt{J_1:x^\a}$. Then we must have

\begin{enumerate}
    \item There exist two triangles $C_1$, $C_2$ with $N[C_1] \cap C_2 = \emptyset$ such that $x^\a = C_1 \cdot C_2 \cdot x^\b$ where $1 \neq x^\b \notin I$.
    \item $\deg f = 1$ and $\supp f \notin \supp \a$.
\end{enumerate}
\end{lem}
\begin{proof} By Lemma \ref{radical_colon}, there exists a minimal generator $P$ of $J_2$ such that $f = \sqrt{P/\gcd(P,x^\a)}$. In particular, 
\begin{equation}\label{eq_3_2}
\supp f \subseteq \supp P \subseteq \supp f \cup \supp \a.
\end{equation}
Since $f \notin \sqrt{J_1 :x^\a}$, we have $P \notin I^4$ and $f \notin I$. By Theorem \ref{expansion_edge}, there are two cases.

\smallskip

\noindent\textbf{Case 1.} $P = x_{C_1} \cdot x_{C_2}$ where $C_1, C_2$ are disjoint cycles of length $3$ and $5$ respectively. By Lemma \ref{lem_radical_power} and \eqref{eq_3_2},
$$\rp_\a(f) \ge 4 \implies f \in \sqrt{I^4:x^\a} \subseteq \sqrt{J_1:x^\a}, \text{ a contradiction}.$$

\smallskip

\noindent\textbf{Case 2.} $P = x_{C_1} \cdot x_{C_2} \cdot x_e$, where $C_1, C_2$ are $3$-cycles such that $N[C_1] \cap C_2 = \emptyset$ and $e$ is an edge of $G$. Assume that $C_1 = 123$, $C_2 = 456$, and $e = uv$, note that $u,v$ might coincide with $\{1, ..., 6\}$. Since $f \notin I$, 
\begin{equation}\label{eq_3_3}
|\supp f \cap \{1,2,3\}| \le 1, |\supp f \cap \{4,5,6\}| \le 1, \text{ and } | \supp f \cap \{u,v\} | \le 1.
\end{equation}
Let $\b, \c$ be the restriction of $\a$ to $\{1,2,3\}$ and $\{4,5,6\}$ respectively. There are several subcases:

Subcase 2.a. $|\{u,v\} \cap \{1,2,3,4,5,6\} | = 2$. Since $N[123] \cap \{4,5,6\} = \emptyset$, we may assume that $u = 5$, $v = 6$. If $\supp f \cap \{1,2,3\} \neq \emptyset$, then $\rp_\b(f) \ge 2$, $\rp_\c(f) \ge 2$. If $\supp f \cap \{4,5,6\} \neq \emptyset$, then $\rp_\b(f) \ge 1$, $\rp_\c(f) \ge 3$. In either cases, $\rp_\a(f) \ge 4$, hence $f\in \sqrt{I^4:x^\a}$, a contradiction

Subcase 2.b. $|\{u,v\} \cap \{1,2,3,4,5,6\}| = 1$. We may assume that $u = 6$, $v = 7$. If $\supp f \cap \{1,2,3\} \neq \emptyset$, then $\rp_\b(f) \ge 2$, $\rp_\c(f) \ge 2$, hence $f \in \sqrt{I^4:x^\a}$, a contradiction. Thus, we must have $\supp f \cap \{1,2,3\} = \emptyset$. If $\supp f \cap \{4,5,6\} \neq \emptyset$, by \eqref{eq_3_3}, there are subcases as follows.

2.b.$\alpha$. If $4 \in \supp f$ or $5 \in \supp f$ then $6\in \supp \c$. Thus $\rp_\c(f) \ge 3$, hence $\rp_\a(f)  \ge 4$, a contradiction.

2.b.$\beta$. If $6 \in \supp f$, let $\mathbf{d}$ be the restriction of $\a$ to $\{4,5,7\}$. Then $\rp_{\mathbf{d}}(f) \ge 3$, hence $\rp_\a(f) \ge \rp_\b(f) + \rp_{\mathbf{d}} (f) \ge 4$, a contradiction. 

Thus, we must have $\supp f \cap \{1,..,6\} = \emptyset$. Therefore, $f = x_7$, and $x_1\cdots x_5 x_6^2 | x^\a$. But 
$$x_7^2 x_6^2 x_4x_5 x_1x_2 \in I^4 \implies f \in \sqrt{J_1:x^\a}, \text{ a contradiction}.$$

Subcase 2.c. $|\{u,v\} \cap \{1,2,3,4,5,6\} | = 0$. We may assume that $u = 7, v=8$. Let $\mathbf{d}$ be the restriction of $\a$ to $\{7,8\}$. If $\supp f \cap \{1,2,3\} \neq \emptyset$, then $\rp_\b(f) \ge 2$. If $\supp f \cap \{4,5,6\} \neq \emptyset$, then $\rp_\c(f) \ge 2$. In either cases, $\rp_\a(f) \ge \rp_\b(f) + \rp_\c(f) + \rp_{\mathbf{d}}(f) \ge 2 + 1 + 1 = 4$, a contradiction. Therefore, $\supp f \cap \{1, ..., 6\} = \emptyset$. Thus $\supp f \subseteq \{7,8\}$. Since $f \notin I$, $\deg f = 1$. We may assume that $f = x_8$. Hence $x_1\cdots x_7 | x^\a$. Furthermore, $8 \notin \supp \a$, as $f = \sqrt{P/\gcd(P,x^\a)}$.

The conclusion follows.
\end{proof}

\begin{lem}\label{fourth_in} Let $I = I(G)$ be the edge ideal of a simple graph $G$. Let $J_1 \subseteq J_2$ be intermediate ideals in $\Inter(I^4, \overline{I^4})$. Then 
$$\reg J_2 \le \reg J_1.$$
\end{lem}
\begin{proof} We prove by induction on $n = |V(G)|$. Let $(\a,i)$ be an extremal exponent of $J_2$. By Lemma \ref{extremal_red}, we may assume that $\Delta_\a(J_2) \neq \Delta_\a(J_1)$. By Lemma \ref{lem_colon_4}, there exists $f \in \sqrt{J_2 : x^\a}$ such that $\deg f = 1$ and $\supp f \notin \supp \a$. Let $K_1, K_2$ be the restriction of $J_1, J_2$ to $V = [n] \setminus \supp f$. By Corollary \ref{restriction_intermediate}, $K_1 \subseteq K_2$ are intermediate ideals in $\Inter(I_V^4,\overline{I_V^4})$. By induction $\reg K_2 \le \reg K_1$. By Lemma \ref{red0} and Lemma \ref{restriction_inq},
$$\reg J_2 = \reg (J_2,f) = \reg(K_2,f) \le \reg(K_1,f) = \reg(J_1,f) \le \reg J_1.$$
The conclusion follows.
\end{proof}

To prove the reverse inequality $\reg(S/I^4) \le \reg (S/ \overline {I^4})$, we also use induction on $n=|V(G)|$. For simplicity of exposition, throughout the rest of this section, we always assume that $(\a,i)\in\NN^n\times\NN$ is an extremal exponent of $I^4$, where $I=I(G)$. It is clear that $x^\a\notin I^4$. By Lemma \ref{extremal_red}, it suffices to consider the cases $x^\a\notin \overline{I^4}$ with $\D_\a(I^{4})\ne\D_\a(\overline{I^4})$ or $x^\a\in \overline{I^4}$. We will see frequently that $\a$ has a good decomposition $\a = \b + \c$. As in the proof of Theorem \ref{third_pow}, we then take $K$ to be a maximum induced matching of $G$ restricted to $\supp \b$. It is fairly simple to show that all the assumptions of Lemma \ref{lem_inequality_special} hold, hence $\reg I^4 \le \reg \overline{I^4}$. Thus, once we show that $\a$ has a good decomposition, we conclude our argument. We treat each form of $\a$ in the following lemmas.

\begin{lem}\label{lem_4_1} Assume that there exist two induced cycles $C_1$, $C_2$ with $N[C_1] \cap C_2 = \emptyset$ of length $3$ and $5$ respectively, such that $x_{C_1}x_{C_2} | x^\a$. Then $\reg I^4 \le\reg \overline {I^4}.$
\end{lem}
\begin{proof} Write $x^\a = x_{C_1} x_{C_2} \cdot f$, where $f$ is a monomial in $S$. Since $x^\a\notin I^4$, 
\begin{equation}\label{eq_4_1}
    N[C_1 \cup C_2] \cap \supp (f)  =\emptyset.
\end{equation} 
Furthermore, $x_i \in \sqrt{I^4:x^\a}$ for all $i \in N[C_1 \cup C_2]$. Thus the decomposition $x^\a = (x_{C_1}x_{C_2}) \cdot f$ is a good decomposition. Hence, $\reg I^4 \le \reg \overline{I^4}$.
\end{proof}

\begin{lem}\label{lem_4_2} Assume that there exist two triangles $C_1, C_2$ with $N[C_1] \cap C_2 = \emptyset$ and an edge $e$ such that $x_{C_1}x_{C_2}x_e | x^\a$. Then $\reg(I^4) \le \reg \overline{I^4}.$
\end{lem}
\begin{proof} Without loss of generality, we may assume that $x^\a = x_1 \cdots x_6 x_ux_v \cdot f$, where $C_1=123$ and $C_2=456$ are triangles with $N[C_1] \cap C_2 = \emptyset$, $uv$ is an edge of $G$ and $f$ is a monomial in $S$. Note that $u,v$ might belong to $\{1, ..., 6\}$. Since $x_i(x_1\cdots x_6) \in I^3$ for all $i \in N[\{1, ..., 6\}]$, and $x^\a \notin I^4$, we have 
\begin{equation}\label{eq_4_2_1}
    x_i \in I^4:x^\a \text{ for all } i \in N[\{1, ...,6\}] \implies N[\{1,...,6\}] \cap \supp f = \emptyset.
\end{equation}

Assume first that $\{u,v\} \cap \{1, ..., 6\} \neq \emptyset$. Say $u = 6$. If $i \in N(v)$, then 
$$(x_1x_2)(x_4x_6)(x_5x_6)(x_vx_i)x_3 \in I^4 \implies x_i \in I^4:x^\a.$$
Together with \eqref{eq_4_2_1} and the fact that $x^\a \notin I^4$, we have
\begin{equation}
    x_i \in I^4:x^\a \text{ for all } i \in N[\supp \b] \implies N[\supp \b ] \cap \supp f = \emptyset,
\end{equation}
where $x^\b = x_1\cdots x_5 \cdot x_6^2 \cdot x_v$. Thus the decomposition $x^\a = x^\b \cdot f$ is a good decomposition, hence $\reg I^4 \le \reg \overline{I^4}$.

We now assume that $\{u,v\} \cap \{1,...,6\} = \emptyset$. Assume that $u = 7, v= 8$. There are two cases.

\smallskip
\noindent\textbf{Case 1.} $N(\{1,...,6\}) \cap \{7,8\}  = \emptyset$. By \eqref{eq_4_2_1}, the decomposition $x^\a = (x_1\cdots x_6) \cdot (x_7x_8 f)$ is a good decomposition, hence $\reg I^4 \le \reg \overline{I^4}$. 

\smallskip

\noindent\textbf{Case 2.} $N(\{1,...,6\}) \cap \{7,8\} \neq \emptyset$. We may assume that $67 \in G$. Since $x^\a \notin I^4$, we deduce that
\begin{equation}\label{eq_4_2_2}
N(\{1,2,3\}) \cap (\{8\} \cup \supp f) = \emptyset.
\end{equation}
We claim that $7 \notin N(\{1,2,3\})$. Assume by contradiction that $7 \in N(\{1,2,3\})$. Since $x^\a \notin I^4$, $a_7 = 1$ and $x_8 \notin N(\{1, ..., 6\})$. Furthermore, 
\begin{equation}\label{eq_4_2_4}
    x_i \in I^4:x^\a \text{ for all } i \in N(\{8\} \cup \supp f).
\end{equation}
By Remark \ref{rem_mingens_degree_complex}, there exists a minimal generator $g$ of $\sqrt{I^4:x^\a}$ such that $x_8 | g$. By \eqref{eq_4_2_1} and \eqref{eq_4_2_4}, $\supp g$ is an independent set and
$$\supp g \cap N(\supp \a \setminus \{7\}) = \emptyset \implies N(\supp g) \cap \supp \a = \{7\}.$$
Let $h = \prod_{u \notin N[\supp g]} x_u^{a_u}.$ Since $8\in \supp g$, $7 \in N(\supp g)$, hence $h | x_1\cdots x_6 \cdot  f$ has $\ord_I(h) \le 2$. Thus
$$\sum_{j \in N(\supp g)} a_j + \ord_I(h) \le 3,$$
a contradiction to Lemma \ref{criterion_in_power}. 

Thus, $7 \notin N(\{1,2,3\})$. By \eqref{eq_4_2_1} and \eqref{eq_4_2_2} the decomposition $x^\a = (x_1x_2x_3) \cdot (x_4\cdots x_8 \cdot f)$ is a good decomposition. Hence, $\reg I^4 \le \reg \overline{I^4}.$
\end{proof}

Recall that $(\a,i)$ is an extremal exponent of $I^4$. We now fix a face $F\in \D_{\a}(I^4)$ such that $F\cap\supp(\a)=\emptyset$ and $\h_{i-1}(\lk_{\D_{\a}(I^4)} F;\k)\ne 0$. 
\begin{lem}\label{lem_4_3} Assume that $x^\a\notin \overline{I^4}$ and $\Delta_\a(I^4) \neq \Delta_\a(\overline{I^4})$. Then $\reg I^4 \le \reg \overline{I^4}$.
\end{lem}
\begin{proof} By Lemma \ref{lem_colon_4}, there exist two triangles $123$ and $456$ with $N[\{1,2,3\}] \cap \{4,5,6\} = \emptyset$ and a monomial $f \neq 1$ such that $\supp f$ is an independent set and $x^\a = x_1 \cdots x_6 \cdot f$. Since $x^\a \notin \overline{I^4}$, we have $|\supp f \cap \{1,...,6\} | \le 1$. First assume that $|\supp f \cap \{1,...,6\} | = 1$, say $x_6 |f$. Write $x^\a = x_1 \cdots x_5 \cdot x_6^t \cdot g$, where $ t \ge 2$. Since $x^\a \notin I^4$, 
\begin{equation}\label{eq_4_3_1}
    x_i \in I^4:x^\a \text{ for all } i \in N[\{1,2,3\}] \implies N[\{1,2,3\}] \cap \supp g = \emptyset.
\end{equation}
Thus the decomposition $x^\a = (x_1x_2x_3) \cdot (x_4x_5x_6^tg)$ is a good decomposition. Hence, $\reg I^4 \le \reg \overline{I^4}$.

Thus, we may assume that $x^\a = x_1 \cdots x_6 \cdot x_7^{a_7} \cdots x_u^{a_u}$, where $a_7, ..., a_u > 0$. There are two cases.

\smallskip

\noindent\textbf{Case 1.} $N(\{1,...,6\}) \cap \{7, ...,u\} \neq \emptyset$. Assume that $67 \in G$. Since $x^\a \notin I^4$, we have 
\begin{equation}\label{eq_4_3_2}
    x_i \in I^4:x^\a \text{ for } i \in N[\{1,2,3\}] \implies N[\{1,2,3\}] \cap \{8, ..., u\} = \emptyset.
\end{equation}
There are two subcases:

Subcase 1.a. $7 \in N(\{1,2,3\})$. Since $x^\a \notin I^4$, $a_7 = 1$. With argument similar to \textbf{Case 2} in the proof of Lemma \ref{lem_4_2}, we see that if $u > 7$, then $\Delta_\a(I^4)$ is a cone over $u$, which is a contradiction. Thus, $x^\a = x_1 \cdots x_7$. Then
$$\sqrt{I^4:x^\a} = (x_i \mid i \in N[\{1,...,6\}]) + I(H),$$
where $H$ is the restriction of $G$ to $[n] \setminus N[\{1,...,6\}]$. Thus $\Delta_\a(I^4) = \Delta(I(H))$. By Lemma \ref{Key0}, $\reg I^4 \le \reg I(H)  + 7$. Let $L= I(H) + K$ be the restriction of $I$ to $[n] \setminus N[\{1,...,6\}] \cup  \{1,2,4,5\}$, where $K = (x_1x_2,x_4x_5)$. Then by induction, $\reg L^4 \le \reg \overline{L^4}$. By Theorem \ref{thm_mixed_sum}, we have 
$$\reg L^4 \ge \reg I(H) + \reg K^4 - 1 = \reg I(H) + 8 > \reg I^4,$$
which is a contradiction to Lemma \ref{restriction_inq}.

Subcase 1.b. $7 \notin N(\{1,2,3\})$. By \eqref{eq_4_3_2}, the decomposition $x^\a = (x_1x_2x_3) \cdot (x_4x_5x_6 x_7^{a_7}\cdots x_u^{a_u})$ is a good decomposition. Hence $\reg I^4 \le \reg \overline{I^4}$.

\smallskip

\noindent\textbf{Case 2.}  $N(\{1,...,6\}) \cap \{7, ...,u\} = \emptyset$. If either of the intersections $N[\{1,2,3\}] \cap N[\{4,5,6\}], N(\{1,...,6\}) \cap N(\{7, ..., u\}), N(i) \cap N(j)$ for $i \neq j \in \{7,...,u\}$ is non-empty, say $r$ belongs to one of them, then $x_r \in \sqrt{I^4:x^\a}$ and $r \notin \supp \a$. Let $J$ be the restriction of $I$ to $[n] \setminus \{r\}$. By Lemma \ref{restriction_inq}, Lemma \ref{red0}, and induction,
$$\reg I^4 = \reg (I^4,x_r) = \reg (J^4,x_r) \le \reg (\overline{J^4},x_r) \le \reg \overline{I^4}.$$
Thus, we may assume that 
\begin{equation}\label{eq_4_3_3}
    N[123] \cap N[456] = \emptyset, N[\{1,...,6\}] \cap N(\{7,...,u\}) = \emptyset, \text{ and } N(i) \cap N(j) = \emptyset,
\end{equation} 
for $i\neq j \in \{7, ..., u\}$. For a set $U \subset [n]$, we denote by $(U) := (x_i \mid i \in U)$ the ideal generated by variables in $U$. The notation $(U)*(V)$ denotes the product of two linear ideals. Since $\ord_I(x_1 \cdots x_6) = 2$, it is clear that

\begin{align} \label{eq_4_3_4}
\begin{split}
 \sqrt{I^4:x^\a} & \supseteq I  + (N[\{1,...,6\}]) * (N(\{7,...,u\}) ) \\
    &+ (N[\{1,2,3\}]) * (N[\{4,5,6\}]) + \sum_{i,j \in \{7,...,u\}, ~ i  < j} (N(i)) * (N (j)).
    \end{split}
\end{align}

We first claim 
\begin{equation}\label{eq_4_3_5}
    a_j = 1 \text { for all } j = 7,\ldots, u.
\end{equation} 
\begin{proof}[Proof of \eqref{eq_4_3_5}]
Assume by contradiction that $a_t > 1$ for some $t\in \{7, ..., u\}$. Without loss of generality, we may assume that $a_j > 1$ for all $j = 7, ..., t$ and $a_j = 1$ for all $j = t+1, ..., u$. Then

\begin{equation}\label{eq_4_3_6}
    x_i \in \sqrt{I^4:x^\a} \text{ for all } i \in N(\{7,...,t\}).
\end{equation}
By Remark \ref{rem_mingens_degree_complex}, there exists a minimal generator $g$ of $\sqrt{I^4:x^\a}$ such that $x_7 | g$. Since $g$ is minimal, by \eqref{eq_4_3_3}, \eqref{eq_4_3_4} and \eqref{eq_4_3_6}, we deduce that $\supp g$ is an independent set and
\begin{equation}\label{eq_4_3_7}
    \supp g \cap N(\{7,...,t\}) = \emptyset \text{ and } |\supp g \cap (N[\{1,...,6\}] \cup N(\{t+1, ..., u\})) | \le 1.
\end{equation} 
In particular, $\sum_{j \in N(\supp g)} a_j \le 1$. Furthermore, $\ord_I(x^\a) = 2$, thus 
$$\sum_{j \in N(\supp g)} a_j + \ord_I \left ( \prod_{u\notin N[\supp g]} x_u^{a_u} \right ) \le 1 + 2 = 3,$$
a contradiction to Lemma \ref{criterion_in_power}.
\end{proof}
Thus $x^\a = x_1 \cdots x_6 \cdot x_7 \cdots x_u$. We now claim that \eqref{eq_4_3_4} is an equality, namely
\begin{align} \label{eq_4_3_8}
\begin{split}
 \sqrt{I^4:x^\a} & = I  + (N[\{1,...,6\}]) * (N(\{7,...,u\}) ) \\
    &+ (N[\{1,2,3\}]) * (N[\{4,5,6\}]) + \sum_{i,j \in \{7,...,u\}, ~ i  < j} (N(i)) * (N (j)).
    \end{split}
\end{align}
\begin{proof}[Proof of \eqref{eq_4_3_8}] Let $x_U$ be a minimal generator of $\sqrt{I^4:x^\a}$. We may assume that $U$ is an independent set. Since $\ord_I(x^\a) = 2$, by Lemma \ref{criterion_in_power}, we must have $\sum_{j\in N(U)} a_j \ge 2$. Furthermore, by assumption \eqref{eq_4_3_3}, we deduce that there are two indices $i,j \in \{1, ..., u\}$ such that $U \cap N(i) \neq 0$ and $U \cap N(j) \neq 0$. Thus $x_U$ belongs to the right hand side.
\end{proof}
We now claim:
\begin{equation}\label{eq_4_3_9}
    u = 7, \text { hence } x^\a = x_1 \cdots x_7.
\end{equation} 
\begin{proof}[Proof of \eqref{eq_4_3_9}] Assume by contradiction that $u > 7$. Let 
\begin{align*}
    J &= I + (N[\{1,2,3\}]) * (N[\{4,5,6\}]) + (N[\{1,...,6\}]) * (N(\{7,...,u\}) ) \\
    &+ \sum_{i,j \in \{7,...,u\}, ~ i  < j, (i,j) \neq (u-1,u)} (N(i)) * (N (j)).
\end{align*}
By \eqref{eq_4_3_8}, $\sqrt{I^4:x^\a} = J + (N(u-1))* (N(u)) = (J + (N(u-1)) \cap (J+(N(u)))$. Let $\Gamma_1$ and $\Gamma_2$ be the simplicial complexes corresponding to $J + (N(u-1))$ and $J + (N(u))$. By Lemma \ref{cone}, $\Delta_\a(I^4) = \Gamma_1 \cup \Gamma_2$ and $\Gamma_1 \cap \Gamma_2$ corresponds to $J + (N(u-1)) + (N(u))$. By Lemma \ref{cone}, $\Gamma_1$ is a cone over $u-1$ and $\Gamma_2$ and $\Gamma_1 \cap \Gamma_2$ are cones over $u$. Since $F \cap \supp \a = \emptyset$, $\lk_{\Gamma_1} F$ is a cone over $u-1$, $\lk_{\Gamma_2} F$ and $\lk_{\Gamma_1 \cap \Gamma_2} F$ are cones over $u$. Looking at the Mayer-Vietoris sequence for simplicial homology, $\lk_{\Delta_\a(I^4)} F$ has trivial homology, a contradiction.
\end{proof}
Thus, $x^\a = x_1 \cdots x_7$. Let $J = I + (N[\{1,2,3]) * (N[\{4,5,6\}])$. Then by \eqref{eq_4_3_8}, $\sqrt{I^4:x^\a} = (J + (N[\{1,\ldots,6\}])) \cap (J + (N(7))).$ Let $\Delta = \Delta_\a(I^4)$ and $\Gamma_1$ and $\Gamma_2$ be the simplicial complexes corresponding to $J + (N[\{1,\ldots,6\}])$ and $J + (N(7))$. Then $\Delta = \Gamma_1 \cup \Gamma_2$. By Lemma \ref{cone}, $\Gamma_2$ and $\Gamma_1 \cap \Gamma_2$ are cones over $7$. Thus $\lk_{\Gamma_2}F$ and $\lk_{\Gamma_1} F \cap \lk_{\Gamma_2}F$ have trivial homology. Applying the Mayer-Vietoris sequence we have 
$$ 0\to  \h_{i-1} (\lk_{\Gamma_1} F;\k) \to \h_{i-1} (\lk_{\Delta}F;\k) \to 0.$$
By Lemma \ref{Key0}, 
$$\reg I^4 = |\a| + i + 1 \le \reg I_{\Gamma_1} + 7.$$
Since $J + (N[\{1,...,6\}]) = I(H) + (N[\{1,...,6\}])$ where $H$ is the restriction of $I$ to $[n] \setminus N[\{1,...,6\}]$. Thus $I_{\Gamma_1} = I(H)$. Let $L$ be the restriction of $I$ to $[n] \setminus N[\{1,...,6\}] \cup \{1,2,4,5\}$. Then $L = I(H) + K$, where $K = (x_1x_2,x_4x_5)$. By induction, $\reg L^4 \le \reg \overline{L^4}$. By Theorem \ref{thm_mixed_sum}, 
$$\reg L^4 \ge \reg I(H) + \reg K^4 - 1 = \reg I(H) + 8 > \reg I^4,$$
a contradiction to Lemma \ref{restriction_inq}. That concludes the proof of the Lemma.
\end{proof}

\begin{thm}\label{fourth_pow} Let $I$ be the edge ideal of a simple graph $G$. Let $J \in \Inter(I^4, \overline{I^4})$ be an intermediate ideal. Then 
$$\reg J = \reg \overline {I^4}  = \reg I^4.$$
\end{thm}
\begin{proof} By Lemma \ref{fourth_in}, it suffices to prove that $\reg I^4 \le \reg \overline{I^4}$. Let $(\a,i)$ be an extremal exponent of $I^4$. There are two cases.

\smallskip

\noindent\textbf{Case 1.} $x^\a \notin \overline{I^4}$. By Lemma \ref{extremal_red}, we may assume that $\Delta_\a(I^4) \neq \Delta_\a(\overline{I^4)}$. By Lemma \ref{lem_4_3}, $\reg I^4 \le \reg \overline{I^4}$.

\smallskip

\noindent\textbf{Case 2.} $x^\a \in \overline{I^4}$. Since $x^\a \notin I^4$, the conclusion follows from Theorem \ref{expansion_edge}, Lemma \ref{lem_4_1}, and Lemma \ref{lem_4_2}.
\end{proof}

We are now ready for the proof of the main theorem.
\begin{proof}[Proof of Theorem \ref{main_thm}] Follows from Theorem \ref{expansion_edge}, Theorem \ref{third_pow}, and Theorem \ref{fourth_pow}.
\end{proof}

\section{Integral closure of powers of Stanley-Reisner ideals of one-dimensional simplicial complexes}\label{sec_dim1}
In this section, we compute the regularity of intermediate ideals in $\Inter(I_\Delta^s,\overline{I_\Delta^s})$ for a one-dimensional simplicial complex $\Delta$.

\begin{thm}\label{dim1} Let $I = I_\Delta$ be the Stanley-Reisner ideal of a one-dimensional simplicial complex $\Delta$. Assume that $I$ is non-normal. Then $\girth \Delta \le 4$. Furthermore, for all $s \ge 1$, and all $J \in \Inter(I^s, \overline{I^s})$, we have 
$$\reg J = \reg I^s = \begin{cases} 3s & \text{ if } \girth \Delta = 3\\
2s + 1 & \text{ if } \girth \Delta = 4.\end{cases}$$
\end{thm}
\begin{proof} If $\girth \Delta \ge 4$ then $I_\Delta$ is the edge ideal of a graph $G$ which is a complement of $\Delta$. By Theorem \ref{expansion_edge}, $I$ is non-normal if and only if $G$ has an induced subgraph consisting of two disjoint odd cycles. Since $\dim \Delta = 1$, these two odd cycles must be triangles. Hence, $\girth \Delta \le 4$. Let $(\a,i)$ be an extremal exponent of $J$. By Lemma \ref{partial_degree} and Remark \ref{rem_extremal_set}, $a_j \le s-1$ for all $j \in [n]$. There are two cases. 

\noindent\textbf{Case 1.} $\girth \Delta = 3$. The arguments in \cite[Lemma 3.2]{MV2} carry over give us $\reg J = 3s$.

\noindent\textbf{Case 2.} $\girth \Delta = 4$. Then $I = I_\Delta$ is the edge ideal of a graph $G$ which is the complement of $\Delta$. The lower bound $\reg J \ge 2s + 1$ follows from Lemma \ref{restriction_inq}. Thus, it suffices to prove that $|\a| +i \le 2s$. The argument in \cite[Lemma 3.3]{MV2} stills apply, and thus if $i = 2$, then $|\a| \le 2s-2$, hence $|\a| + i \le 2s$. Now assume that $i \le 1$, thus we may assume that $|\a| \ge 2s-1$. Since $I^s \subseteq J \subseteq \overline{I^s} \subseteq I^{(s)}$, we have
$$\sqrt{I^s:x^\a} \subseteq \sqrt{J:x^\a} \subseteq \sqrt{I^{(s)}:x^\a} = \sqrt{I^s:x^\a}$$
where the last equality follows from Lemma \ref{Key0} and \cite[Lemma 3.5]{MV2}. By Lemma \ref{Key0}, $\Delta_\a(J) = \Delta_\a(I^s)$. By Lemma \ref{extremal_red} and \cite[Theorem 1.1]{MV2}, we deduce that 
$$\reg J \le \reg I^s = 2s+1.$$
That concludes the proof of the Theorem.
\end{proof}

\begin{exm} Let 
\begin{align*}
    I = & (x_1x_2,x_1x_3,x_2x_3, x_4x_5,x_4x_6,x_5x_6,x_3x_7,x_1x_4x_7,x_2x_4x_7,x_1x_5x_7,x_2x_5x_7,\\
    & x_6x_7,x_1x_8,x_2x_4x_8,x_3x_4x_8,x_5x_8,x_2x_6x_8,x_3x_6x_8,x_2x_7x_8,x_4x_7x_8) \subseteq \k[x_1,...,x_8].
\end{align*}
It is easy to check with Macaulay2 \cite{M2} that $\dim \Delta(I) = 1$. Furthermore, $I$ is non-normal, as $f = x_1\cdots x_6 \in \overline{I^3}$, but $f \notin I^3$. For each $s \ge 0$, let $f_1 = f \cdot (x_1x_8)^s, f_2 = f\cdot  (x_2x_6x_8)^s$. Then $f_1, f_2$ are minimal generators of $\overline{I^{s+3}}$. By Theorem \ref{dim1}, 
$$\reg \left ( I^{s+3} + (f_1) \right )  = \reg \left ( I^{s+3} + (f_1,f_2) \right ) = 3s + 9,$$
for all $s \ge 0$.
\end{exm}

There are also plenty of examples of non-normal edge ideals of graphs $I(G)$ with $\dim \Delta(I) = 1$, and $G$ has more than two odd cycles.
\begin{exm} Let
\begin{align*}
    I = & (x_1x_2,x_1x_3,x_2x_3, x_4x_5,x_4x_6,x_4x_7,x_5x_6,x_5x_7,x_6x_7,x_3x_8,x_4x_8,x_5x_8,x_6x_8,x_7x_8,\\
    & x_1x_9,x_2x_9,x_3x_9,x_4x_9,x_7x_9,x_1x_{10},x_2x_{10}, x_3x_{10},x_5x_{10},x_6x_{10},x_8x_{10}) \subseteq \k[x_1,...x_{10}].
\end{align*}
It is easy to check with Macaulay2 \cite{M2} that $\dim \Delta(I) = 1$. For each $s \ge 0$, let $f_1 = x_1\cdots x_6 (x_8x_{10})^s$, $f_2 = x_1\cdots x_5 x_7 (x_4x_9)^s$, $f_3 = x_1\cdots x_4 x_6x_7 (x_7x_8)^s$. Then $f_1, f_2,f_3$ are minimal generators of $\overline{I^{s+3}}$. By Theorem \ref{dim1}, 
$$\reg \left ( I^{s+3} + (f_1,f_2) \right ) = \reg \left ( I^{s+3} + (f_1,f_3) \right ) = 2s+7,$$
for all $s \ge 0$.
\end{exm}

\section{Characteristic dependence of regularity of powers}\label{sec_ex}
In this section, we provide an example of an ideal whose regularity of powers depends on the characteristic of the base field. We first give a general statement which is a simple consequence of \cite[Theorem 1.1]{NV1}. 
\begin{thm}\label{thm_ex} Let $H$ be a graph whose edge ideal $I(H)$ satisfies the condition
$$\reg I(H)^s \le \reg I(H) + 2s - 2 \text{ for all } s \ge 1.$$
Let $T$ be a graph whose edge ideal $I(T)$ satisfies the condition 
$$\reg I(T)^s = \reg I(T) + 2s-2 \text{ for all } s \ge 1.$$ Let $G = H \cup T$ be the disjoint union of $H$ and $T$. Then for all $s \ge 1$ we have 
$$\reg I(G)^s = \reg I(H) + \reg I(T) + 2s - 3.$$
\end{thm}
\begin{proof}
By Theorem \ref{thm_mixed_sum}, we have 
$$\reg I(G)^s = \max_{i\in[1,s-1],j\in[1,s]} (\reg I(H)^i + \reg I(T)^{s-i}, \reg I(H)^j + \reg I(T)^{s-j+1} - 1).$$
For all $1 \le i \le s-1$, we have  
\begin{align*}
    \reg I(H)^i + \reg I(T)^{s-i} &\le \reg I(H) + 2i- 2 + \reg I(T) + 2(s-i)-2\\
    & < \reg I(H) + \reg I(T) + 2s-3.
\end{align*} 
For all $2 \le j \le s$, we have 
\begin{align*}
    \reg I(H)^j + \reg I(T)^{s-j+1} &\le \reg I(H) + 2j - 2 + \reg I(T) + 2(s-j) \\
    & = \reg I(H) + \reg I(T)+2s-2.
\end{align*} 
Finally, 
$$\reg I(H) + \reg I(T)^s = \reg I(H) + \reg I(T) + 2s- 2.$$
The conclusion follows.
\end{proof}

\begin{cor}\label{cor_char_dependence} Let $G = C_1 \cup C_2 \cup H$ be the disjoint union of two triangles and a bipartite graph $H$ whose $\reg I(H)$ depends on the characteristic of the base field (e.g. the Dalili-Kummini's example \cite{DK}). Then for all $s \ge 1$, we have 
$$\reg I^s = \reg I^{(s)} = \reg \overline{I^s} = \reg I(H) +2s.$$
In particular, $\reg I^s, \reg \overline{I^s}, \reg I^{(s)}$ depend on the characteristic of the base field as $\reg H$ does.
\end{cor}
\begin{proof}By Theorem \ref{thm_ex} 
$$\reg I^s = \reg I(H) + \reg I(C_1 \cup C_2) + 2s-3 = \reg I(H) + 2s.$$
With similar argument as in the proof of Theorem \ref{thm_ex}, by \cite[Theorem 5.11]{HNTT}, and the fact that $I(H)^s = I(H)^{(s)}$, $\reg I(C_1)^{(s)} = \reg I(C_2)^{(s)} = 2s$ for all $s \ge 1$, we deduce that 
$$\reg I^{(s)} = \reg I(H) + \reg I(C_1 \cup C_2) + 2s-3 = \reg I(H) + 2s.$$
By \cite[Theorem 4.14]{KK}, we have $\reg I^s = \reg \overline{I^s}$. The Corollary follows.
\end{proof}

\begin{rem} Let $H$ be the bipartite graph in Dalili-Kummini's example. Computation with Macaulay2 \cite{M2} shows that $\reg I(H)^2 = 6$ does not depend on the characteristic of the base field. 

Let $G$ be the union of $H$ and an edge. Then $G$ is a bipartite graph. By Theorem \ref{thm_ex}, for all $s \ge 1$, $\reg I(G)^s = \reg I(H) + 2s - 1$ depend on the characteristic of the base field.

To get an example of a connected graph with the same phenomenon as in Corollary \ref{cor_char_dependence}, let $P$ be the fiber product of the edge ideal $I(G)$ in Corollary \ref{cor_char_dependence} with any edge ideal $I(L)$ of a bipartite graph such that $\reg I(L)^s < 2s + 4$. By \cite[Corollary 5.2]{NV2}, \cite[Lemma 6.2]{MV1}, and \cite[Theorem 4.14]{KK}, we deduce that $\reg P^s = \reg P^{(s)} = \reg \overline{P^s} = \reg I(G)^s$ depend on the characteristic of the base field.
\end{rem}

\begin{rem}[Katzman's example]Let $I\subseteq \k[x_1,\ldots,x_{11}]$ be the following edge ideal:
\begin{gather*}
I=(x_1x_2, x_1x_6, x_1x_7, x_1x_9, x_2x_6, x_2x_8, x_2x_{10}, x_3x_4, x_3x_5, x_3x_7, x_3x_{10},\\
x_4x_5, x_4x_6, x_4x_{11}, x_5x_8, x_5x_9, x_6x_{11}, x_7x_9, x_7x_{10}, x_8x_9, x_8x_{10}, x_8x_{11}, x_{10}x_{11}).
\end{gather*}
Computation with Macaulay2 \cite{M2} shows that $\reg I^2 = 5$ does not depend on the characteristic of the base field. It is interesting if we can prove that 
$$\reg I^s \le \reg I + 2s - 2,$$
for all $s \ge 2$. If this is true, then the mixed sum $P$ of $I$ with the edge ideal of a triangle gives an example of a non-normal edge ideal on $14$ variables with 
$$\reg P^s = \begin{cases} 3 + 2s & \text{ if } \charr \k = 2 \\
2 + 2s & \text{ if } \charr \k \neq 2.\end{cases},$$
for all $s \ge 1$. 
\end{rem}

\end{document}